\title{\bf{Finite determinacy of matrices and ideals}}
\author{
       \bf{ Gert-Martin Greuel and Thuy Huong Pham}\\
 }
\date{}
\documentclass[11pt]{article}
\usepackage{amsmath, amsthm, amssymb, mathrsfs, amscd, amsthm, amscd,amsfonts,graphicx}
\usepackage{indentfirst,url,xypic}
\usepackage{lipsum} 
\usepackage[all]{xy}
\usepackage{url}
\usepackage[colorlinks,plainpages]{hyperref}
\usepackage{verbatim}
\usepackage{fancyvrb}
\usepackage{listings}
\usepackage[colorlinks,plainpages]{hyperref}
\usepackage{geometry}
\hypersetup{
	colorlinks=true,
	linkcolor=blue,
	filecolor=magenta,      
	urlcolor=cyan,
}

\DeclareMathOperator{\rank}{rank}

\DeclareMathOperator{\mng}{mng}

\DeclareMathOperator{\height}{ht}
\DeclareMathOperator{\depth}{depth}
\DeclareMathOperator{\variety}{V}

\DeclareMathOperator{\characteristic}{char}
\DeclareMathOperator{\order}{ord}
\DeclareMathOperator{\annihilator}{Ann}
\DeclareMathOperator{\spectrum}{Spec}

\DeclareMathOperator{\Ann}{Ann}

\newtheorem{Definition}{ Definition}[section]
\newtheorem{Theorem}[Definition]{Theorem}
\newtheorem{Remark}[Definition]{Remark}
\newtheorem{Proposition}[Definition]{Proposition }
\newtheorem{Corollary}[Definition]{Corollary}
\newtheorem{Example}[Definition]{Example}
\newtheorem{Lemma}[Definition]{Lemma}
\newtheorem{Problem}[Definition]{Problem}
\newtheorem{Conjecture}[Definition]{Conjecture}

\newcommand{\R}{\mathbb{R}}

\newcommand{\N}{\mathbb{N}}
\newcommand{\C}{\mathbb{C}}

\newcommand{\A}{\mathbb{A}}

\begin{document}
\maketitle

\begin{abstract} 

The main aim of this paper is to characterize ideals $I$ in the power series ring $R=K[[x_1,\ldots,x_s]]$ that are finitely determined up to contact equivalence by proving that this is the case if and only if $I$ is an isolated complete intersection singularity, provided dim$(R/I) > 0$ and $K$ is an infinite field (of arbitrary characteristic). Here two ideals $I$ and $J$ are contact equivalent if 
the local $K$--algebras $R/I$ and $R/J$ are isomorphic. If $I$ is minimally generated by $a_1,\ldots,a_m$, we call $I${ \em finitely contact determined} if it is contact equivalent to any ideal $J$ that can be generated by  $b_1,\ldots,b_m$ with $a_i - b_i \in \langle x_1,\ldots,x_s\rangle^k$ for some integer $k$. 
We give also computable and semicontinuous determinacy bounds.


The above result is proved by considering left--right equivalence on the ring 
$M_{m,n}$ of $m\times n$ matrices $A$ with entries in $R$ and we show that the Fitting ideals of a finitely determined matrix in $M_{m,n}$ have maximal height, a result of independent interest. The case of ideals is treated by considering 1-column matrices. Fitting ideals together with a special construction are used to prove the characterization of finite determinacy for ideals in $R$. 
Some results of this paper are known in characteristic 0, but they need new (and more sophisticated) arguments in positive characteristic partly because the tangent space to the orbit of the left-right group cannot be described in the classical way. In addition we point out several other oddities, including the concept of specialization for power series, where the classical approach (due to Krull) does not work anymore.  We include some open problems and a conjecture.

\end{abstract}

\section{Introduction}\label{introduction}
Throughout this paper let $K$ denote a field of arbitrary characteristic and 
\[R:=K[[{\bf{x}}]]=K[[x_1,\ldots, x_s]]\] 
the formal power series ring over $K$ in $s$ variables with maximal ideal $\mathfrak{m} = \langle x_1, \ldots , x_s \rangle$. We denote by 
\[M_{m,n}:=Mat(m,n, R)\]
the set of all $m\times n$ matrices ($m$ rows and $n$ columns) with entries in $R$. 
Two matrices $A, B \in M_{m,n}$ are called {\it left--right equivalent}, denoted  $A\mathop\sim\limits^{G} B$, if they belong to the same orbit of the left--right group 
\[G:=(GL(m,R)\times GL(n,R)^{op}) \rtimes \mathcal{R},\]
\noindent where $G^{op}$ is the opposite group of a group $G$ and 
$\mathcal{R}:=Aut(R)$. The group $G$  acts on $M_{m,n}$ by
\[(U, V, \phi, A)\mapsto U\cdot\phi(A)\cdot V,\]
with  $U\in GL(m,R), \ V\in GL(n,R), \ \phi\in Aut(R)$ and if $A=[a_{ij}({\bf x})]$ then 
 $\phi(A)=[a_{ij}\left(\phi({\bf x})\right)]$ with $\phi({\bf x})=\left(\phi(x_1),\ldots,\phi(x_s) \right)$. 
A matrix $A\in  M_{m,n}$ is said to be {\it $G$ $k$--determined} if for each matrix $B\in  M_{m,n}$ with $B-A\in \mathfrak{m}^{k+1} M_{m,n}$ we have $B\mathop\sim\limits^{G} A$, $\mathfrak{m} = \langle x_1,\ldots,x_s\rangle$ the maximal ideal of $R$. 
$A$ is called {\it finitely $G$--determined} if there exists a positive integer $k$ such that it is $G$ $k$--determined.

Finite determinacy for analytic map-germs and ideals with respect to various equivalence relations has been intensively studied by e.g.  \cite{Hi65}, \cite{Tou68}, \cite{Mat68}, \cite{Gaf79}, \cite{Ple80}, \cite{Dam81}, \cite{Wal81},  \cite{BdPW87}, \cite{CS97}, and many more, mainly with the aim to find sufficient conditions for finite determinacy.   
 In \cite{BK16}, the authors study finite determinacy for matrices with entries in the convergent power series ring $\C\{{\bf x}\}$, without giving determinacy bounds. In  \cite {Pha16} and \cite {GP16} the authors started the study of finite determinacy for matrices of power series over fields of arbitrary characteristic. The case of one power series, i.e. $m=n=1$, is classical over the complex and real numbers. It was treated over a field of arbitrary characteristic in \cite {GK90} for contact equivalence and in \cite {BGM12} for right and contact equivalence.

Let us give a rough overview of the main results, for more detailed statements we refer to the main body of the paper. 
For $A\in \mathfrak{m}\cdot M_{m,n}$ we define the following submodules of $M_{m,n}$,
  \begin{align*}
  \tilde T_A(GA)&:=\langle E_{m, pq}\cdot A\rangle +\langle A\cdot E_{n, hl}\rangle+ \mathfrak{m}\cdot\left\langle\frac{\partial A}{\partial x_\nu}\right\rangle \hskip 6pt{\text{resp.}}\\
  \tilde T^e_A(GA)&:=\langle E_{m, pq}\cdot A\rangle +\langle A\cdot E_{n, hl}\rangle+ \left\langle\frac{\partial A}{\partial x_\nu}\right\rangle,
  \end{align*} 
and call $\tilde T_A(GA)$ resp. $\tilde T^e_A(GA)$ the {\it tangent image} resp. {\em extended tangent image} at $A$ to the orbit $GA$.
  Here $\langle E_{m, pq}\cdot A\rangle$ is the $R$-submodule generated by  $E_{m, pq}\cdot A$, $p,q=1,\ldots,m$, with $E_{m,pq}$ the $(p,q)$-th canonical matrix of $Mat(m,m,R)$ (1 at place $(p,q)$ and 0 else) and $\left\langle\frac{\partial A}{\partial x_\nu}\right\rangle$ is the $R$-submodule generated by the matrices $\frac{\partial A}{\partial x_\nu} =\left [\frac{\partial a_{ij}}{\partial x_\nu} ({\bf x})\right], \nu = 1, \ldots, s$. The submodules $\tilde T_A(GA)$ and $\tilde T^e_A(GA)$ were introduced in  \cite{GP16} and it is proved in  \cite[Proposition 2.5]{GP16} that $\tilde T_A(GA)$ is the image of the tangent map of the orbit map $G \to GA$. Moreover, the following is proved in the same paper, with $\order(A)$ the minimum of the orders of the entries of $A$.

 \begin{Theorem}\rm(\cite[Theorem 3.2]{GP16})\label{GP16}
 		\begin{enumerate}
		\item Let $A=[a_{ij}]\in \mathfrak{m}\cdot M_{m,n}$. If there is an integer $k\ge 0$ such that
 	\begin{align*}
 	\mathfrak{m}^{k+2}\cdot M_{m,n}\subset \mathfrak{m}\cdot \tilde T_A(GA)\tag{1.1}\label{main},
 	\end{align*}
 	then $A$ is $G$ $(2k-\order(A)+2)$-determined. Moreover, \eqref{main} holds 
	for some $k$ iff $\tilde T_A(GA)$ (equivalently $\tilde T^e_A(GA)$) is of finite $K-$codimension in $M_{m,n}$.
	\item If $\characteristic(K)= 0$ then the condition \eqref{main} for some $k$ is {\em equivalent} to $A$ being finitely $G$--determined.	\end{enumerate}
 \end{Theorem}
  
 \begin{Remark} \rm 
 	Condition \eqref{main} gives a {\em sufficient} criterion for finite determinacy of $A$ in any characteristic. We do not know whether, for arbitrary $m$ and $n$,  the finite codimension of the tangent image $\tilde T_A(GA)$ in $M_{m,n}$ is necessary in positive characteristic. On the other hand, the finite codimension of the tangent space $T_A(GA) \subset M_{m,n}$ is easily seen to be necessary in any characteristic (cf. \cite [Lemma 2.11] {GP16}). 
If $\characteristic(K)= 0$ then  $\tilde T_A(GA)$ coincides $T_A(GA)$ (\cite[Lemma 2.8] {GP16}) and hence condition \eqref{main} for some $k$ is necessary and sufficient.

In general we have $\tilde T_A(GA)\subset T_A(GA)$, and it is important to notice that the inclusion can be strict in positive characteristic (see Example 2.9 in \cite{GP16}). 
 \end{Remark}
 
\begin {Conjecture} \rm 
Statement {\it 2.} of Theorem  \ref{GP16} holds also in positive characteristic.
 \end {Conjecture} 
 
One of the main aims of this paper is to prove the conjecture for
matrices in $M_{m,1}$ (c.f. Theorem \ref{column matrix}, where we give also determinacy bounds). 
\begin{Theorem} \label{th1.2}
{\it	For $A \in  \mathfrak{m} \cdot Mat(m,1,R)$ with $K$ infinite, the following are equivalent:
	\begin{enumerate}
		\item $A$ is finitely left--right determined.
		\item $\dim_K\left(M_{m,1}\big/\tilde T_A(GA)\right) <\infty.$
	\end{enumerate} }
$K$ infinite is only needed for (1) $\Rightarrow$ (2) and for $m<s$ (= number of variables).
\end{Theorem}

 Theorem \ref{th1.2} was proved by Mather in \cite[Theorem 3.5] {Mat68} for real and complex analytic maps and  in \cite{CS97} in a slightly more general setting for char($K$) = 0, using classical methods.
The proof  of the interesting direction {\em 1.} $\Rightarrow$  {\em 2.} for $\characteristic(K)> 0$ resisted previous attempts and is more involved as one might think. It requires a special deformation (unknown to us for arbitrary $m,n$) and the above mentioned result about Fitting ideals. \\

Consider now the ideals  $I$ resp. $J$ generated by the entries of $A$ resp. $B \in M_{m,1}$. Then $A$ is left--right equivalent to $B$ iff $I$ and $J$ are {\it contact equivalent} (c.f.  Proposition \ref{contact equivalence}).
Recall that $R/I$ is called a {\it complete intersection} if $\dim(R/I) = s - \mng(I)$ with $\mng(I)$ the minimal number of generators of $I \subset R$. The complete intersection $R/I$  is called an {\it isolated complete intersection singularity (ICIS)} if the ideal of the singular locus,
$I+I_m\left(\left[\frac{\partial f_i}{\partial x_j}\right]\right)$, contains a power of the maximal ideal, where $\{f_1,\ldots, f_m\}$ is a minimal set of generators of $I$ and $I_t(A)$ denotes the ideal of $t\times t$ minors of the matrix $A$.
Theorem  \ref{th1.2}  implies  
\begin {Theorem} {\label{theo1.5}}
	Let $I$ be a proper ideal of $R$. 
	\begin {enumerate}
	\item  If $\dim(R/I) = 0$ then $I$ is finitely contact determined.
	\item  If  $\dim(R/I) >0$ and $K$ is infinite, then $I$ is finitely contact determined if and only if $R/I$ is an isolated complete intersection singularity.
	\end {enumerate} 
\end {Theorem} 	
	
See Theorem \ref{geomchar} for a more detailed statement which contains also the determinacy bound $(2\tau(I)-\order(I)+2)$, where $\tau(I)$ is the Tjurina number of $I$.

Finite determinacy of $I$ implies that $R/I$ is {\em algebraic}, i.e. $R/I \cong R/J$ where $J$ is generated by polynomials.  But it is much stronger since by a result of Artin (\cite[Theorem 3.8]{Ar69})  every isolated singularity $R/I$ (not necessarily an {\em ICIS}) is algebraic. 

\begin {Problem} \rm
The assumption that $K$ is infinite in 2. is only needed to show that a finitely determined $I$ defines an $ICIS$. It is due to our method of proof but we do not know whether it is necessary. For hypersufaces however we show in Theorem \ref{hypersurface} that it is not necessary. 
\end {Problem}


To prove our results we derive a necessary condition for finite $G$-determinacy for matrices in section \ref{necessary}, Theorem \ref{height}, by showing that the Fitting ideals of a finitely $G$-determined matrix have maximal height. For this we use the specialization of ideals depending on parameters, which was introduced by W. Krull and then extended and systematically studied by D.V. Nhi and N.V. Trung for finitely generated modules over polynomial rings and localizations thereof (cf. \cite{NT99}, \cite{NT00}).

\begin {Problem} \rm
A  satisfactory theory for specialization of ideals in power series rings depending on parameters has not yet been developed.
We show in Example \ref{osgood}  that a straightforward generalization of specialization from (localization of) polynomial rings to power series rings does not work.  In Remark \ref{specialization} we propose an approach which is reasonable for uncountable fields $K$ e.g. for $\R, \C$. For a concrete open problem see Problem \ref{pr2.6}.
\end {Problem}

In section \ref {special case} we study $G$--equivalence for 1-column matrices and use the results of section \ref{necessary} to prove Theorem \ref{th1.2}. 
We need and prove a semicontinuity result for modules over a power series ring depending on parameters (Proposition \ref{semi-continuity}) which should be well known, but for which we could not find a reference.
In section  \ref {complete intersections} we apply the results of  section \ref {special case} to contact equivalence for ideals and prove Theorem \ref{geomchar}. Finally we prove in Theorem \ref{hypersurface} that a power series $f \in K[[{\bf x}]]$ is finitely contact (resp. right) determined iff the Tjurina number (resp. the Milnor number) of $f$ is finite, also in positive characteristic.

\section{A necessary finite determinacy criterion by Fitting ideals}\label{necessary}

In this section, we establish a necessary condition for finite $G$-determinacy of matrices in $M_{m,n}$.
Without loss of generality, we assume that $n\le m$. 

For a matrix $A\in Mat(m,n, P)$,  $P$ a commutative Noetherian ring, and an integer  $t$, let $I_t(A)$ denote the ideal of $P$ generated by all $t\times t$ minors of $A$, also called the ($m-t$)-th Fitting ideal of the cokernel of the map
$A: R^n \to R^m$.  

Let ${\bf u}= (u_1,\ldots, u_r)$ be a new set of indeterminates. For an ideal $I \subset K({\bf u})[{\bf x}]$ and $a\in K^r$  the ideal
$$I_a:=\ \{f(a,{\bf x})\ | \ f({\bf u},{\bf x})\in I\cap K[{\bf u}][{\bf x}] \},$$ 
is called the {\em specialization} of $I$  and for an ideal $J \subset K[{\bf u},{\bf x}]$, 
 $$J^e :=J\cdot K({\bf u})[{\bf x}]$$ 
 denotes the {\em extension} of $J$. 
We say that a property holds for {\em generic} $a \in K^r$ if there exists a non-empty Zariski open set $U\subset K^r$ such that the considered property holds for all $a\in U$. 

\begin{Lemma}\label{height of determinantal ideal}
	 Let $K$ be infinite and $M=[g_{ij}({\bf u},{\bf x})]\in Mat(m, n, K[{\bf u}, {\bf x}])$. For $a\in K^r$ set $M_a=[ g_{ij}(a,{\bf x})]\in Mat(m, n,K[{\bf x}])$.
Then for generic $a\in K^r$ and for all $t=1,\ldots, n$ the following holds:
\begin{enumerate}
\item [(i)] $I_t(M_a)=(I_t(M)^e)_a,$ 	
\item [(ii)]  if  $I_t(M_a)$ is a proper ideal then $I_t(M)^e$ is proper too,
 \item [(iii)] if $I_t(M_a)$ is proper then $\height \left(I_t(M_a)\right)=\height (I_t(M)^e), \,  \depth \left(I_t(M_a)\right)= \depth (I_t(M)^e).$
\end{enumerate}
\end{Lemma}

\begin{proof}
$(i)$	For all $a\in K^r$, $I_t(M_a)$ is the ideal of $K[{\bf x}]$ generated by $d_1^{(t)}(a, {\bf x}),\ldots, d_{l_t}^{(t)}(a, {\bf x})$, where $d_1^{(t)}({\bf u}, {\bf x}), \ldots,  d_{l_t}^{(t)}({\bf u}, {\bf x})$ are the $t\times t$ minors of $M$. On the other hand, for each generator $f^{(t)}({\bf u},{\bf x})$ of the ideal $I_t(M)^e\cap K[{\bf u},{\bf x}]$ of $K[{\bf u},{\bf x}]$, there is a polynomial $b^{(t)}({\bf u})\in K[{\bf u}]\setminus \{0\}$ such that $b^{(t)}\cdot f^{(t)}\in I_t(M)$. Therefore, for $a\in K^r$,  which is not a zero of any $b^{(t)}({\bf u})$, we have that $(I_t(M)^e)_a$ is generated by $d_1^{(t)}(a, {\bf x}), \ldots,  d_{l_t}^{(t)}(a, {\bf x})$. Hence, the first assertion holds. \\
$(ii)$	 If  $I_t(M)^e$ is not proper then $K[{\bf x}]=(I_t(M)^e)_a= I_t(M_a)$ by 
$(i)$, contradicting the assumption.\\
$(iii)$ The other statements follow from $(i), (ii)$ and \cite[Theorem 3.4 (ii) and Corollary 4.4]{NT99}.	
\end{proof}

Specialization and extension is also used in the proof of the following theorem which shows that we can modify a matrix with polynomial entries by adding polynomials of arbitrary high order such that the Fitting ideals of the modified matrix have maximal height. The proof was communicated to the authors by  Ng{\^o}~Vi\d{\^e}t Trung in \cite{Tru15} for $t=n$. Using his arguments, we prove for arbitrary $t$:

\begin{Theorem}\label{generic determinantal ideals}
	Let $A=[f_{ij}]\in Mat(m, n, K[{\bf x}])$, $f_{ij}\in \langle x_1,\ldots, x_s\rangle\cdot K[{\bf x}]$ with $K$ infinite. For $N\ge 1$ let  $B=[g_{ij}]\in Mat(m, n, K[{\bf x}])$ be the matrix with entries of the form
	$$g_{ij}=\sum\limits_{k=1}^s a_{ijk}x_k^N, \ a_{ijk}\in K.$$
	Then, for $t\in \{1,\ldots, n\}$ and generic $(a_{ijk})$ in $K^{mns}$, we have with  $m_t:=(m-t+1)(n-t+1)$
	
	$$\height (I_t(A+B))=\min\{s,m_t\}.$$
	If $m_t < s$ then $K[{\bf x}] / I_t(A+B)$ is Cohen-Macaulay and moreover, if $K$ is algebraically closed, then $I_t(A+B)$ is a prime ideal. 
\end{Theorem}

\begin{proof}
	For  $i=1,\ldots, m$, $j=1,\ldots, n$, let
	$$F_{ij}=\sum\limits_{k=1}^s u_{ijk}x_k^N+f_{ij},$$
	where ${\bf u}=\{u_{ijk}\mid i=1,\ldots, m, j=1,\ldots,n, k=1,\ldots, s\}$ is a set of new indeterminates. Set $S=K[{\bf u}, {\bf x}]$ and  let 
	$$M=[F_{ij}]\in Mat(m, n, S).$$ 
	The main work is done to prove the following
	\vskip 4pt
	{\bf Claim:} For all $t=1,\ldots, n$, $I_t(M)^e$ is a proper ideal of $K({\bf u})[{\bf x}]$ and
	$$\height (I_t(M)^e)=\min\{s, m_t\}.$$
	
	Indeed, the first assertion of the claim follows from Lemma \ref{height of determinantal ideal}. Now we prove the second statement. Fix $t\in\{1,\ldots, n\}$. For every $k=1,\ldots, s$ and for all $i, j$
	we have that in $S\left[\frac{1}{x_k}\right]$ $$\frac{1}{x_k^N}F_{ij}=u_{ijk}+\frac{1}{x_k^N}\left(\sum\limits_{h\ne k} u_{ijh}x_h^N+f_{ij}\right).$$
	Therefore, the elements $\frac{1}{x_k^N}F_{ij}$, $k=1,\ldots, s$, $i=1,\ldots,m$, $j=1,\ldots, n$ are algebraically independent over $K\left[{\bf u'}, {\bf x}, \frac{1}{x_k}\right]$, where ${\bf u'}={\bf u}\smallsetminus\{u_{ijk} \, |\, i=1,\ldots, m, j=1,\ldots, n\}$. This means that $\frac{1}{x_k^N}\cdot M$ is a generic matrix over the ring $K\left[{\bf u'}, {\bf x}, \frac{1}{x_k}\right]$ (generic in the sense that the entries are indeterminates, not to be confused with generic points). Note that 
		\[K[{\bf u'}] [{\bf x}]\left[\frac{1}{x_k}\right]\left[\frac{1}{x_k^N} F_{ij} | \  i= 1,\ldots ,m, \ j=1,\ldots, n\right] = S\left[\frac{1}{x_k}\right].\]
	It is well known that the determinantal ideals of a generic matrix are prime and have maximal height (cf. e.g. \cite[(2.13) and (5.18)]{BV88}).
	Hence, $I_t\left(\frac{1}{x_k^N}\cdot M\right)$, the ideal of $S\left[\frac{1}{x_k}\right]$ generated by all $t\times t$ minors of $\frac{1}{x_k^N}\cdot M$, is a prime ideal of the height $m_t$. 
	
	Since
	$I_t(M)\cdot S\left[\frac{1}{x_k}\right]=I_t\left(\frac{1}{x_k^N}\cdot M\right),$ we have that
	$I_t(M)\cdot S\left[\frac{1}{x_k}\right]$ is a prime ideal of height $m_t$. 
This implies that $I_t(M)$ has a prime component, say $P^{(t)}_k$, which does not contain $x_k$, and all other associated primes  must contain $x_k$. 
	
	Let now $k'\in\{1,\ldots, s\}$ and $k'\ne k$. By a similar argument,  $I_t(M)\cdot S\left[\frac{1}{x_k\cdot x_{k'}}\right]$ is a prime ideal of $S\left[\frac{1}{x_k\cdot x_{k'}}\right]$ . It follows that
	$I_t(M)$ has a prime component, say $P^{(t)}_{k,k'}$, which does not contain $x_k$ and $x_{k'}$, and all other associated primes must contain the product $x_k\cdot x_{k'}.$ Therefore, $P^{(t)}_k= P^{(t)}_{k'}$ for all $k\ne {k'}$. Let $P_t$ denote this prime component. Then $P_t$ does not contain any $x_k$, $k=1,\ldots, s$, and all other associated primes of $I_t(M)$ must contain all $x_1,\ldots, x_s.$ Let $Q_t$ be the intersection of all primary components of $I_t(M)$ whose associated primes contain  $x_1,\ldots, x_s.$ Note that $Q_t=S$ if such components do not exit. Then 
	$$I_t(M)=P_t\cap Q_t.$$
	Moreover, since $P_t$ is the only associated prime of $I_t(M)$ which does not contain any $x_k$, 
	$$\height (P_t)=\height \left(I_t(M)\cdot S\left[\frac{1}{x_k}\right]\right)=m_t.$$
	
	We have 
	$$I_t(M)^e=\left(P_t\right)^e\cap (Q_t)^e.$$
	Let first $t$ be such that $s\le m_t.$ We consider two cases:
	
	{\textit {Case 1:}} $\left(P_t\right)^e$ is strictly contained in $ K({\bf u})[{\bf x}]$. Then 
	$$\height \left(\left(P_t\right)^e\right)=\height (P_t)=m_t\ge s,$$
	so that $\height \left(\left(P_t\right)^e\right)=s$.	
	On the other hand, since all other associated primes of $I_t(M)$ contain $x_1,\ldots, x_s$ if they exist, all associated primes of $\left(Q_t\right)^e$ have the height $s$. Hence, in this case $\height (I_t(M)^e)=s.$
\vskip 5pt
	{\textit {Case 2:}} $(P_t)^e=K({\bf u})[{\bf x}]$. Then 
	$I_t(M)^e=(Q_t)^e$
	 so that $\height (I_t(M)^e)=s.$
\vskip 5pt
\noindent	Let now  $t$ be such that $s>m_t$. 	In this case $P_t$ has the least height among the associated primes of $I_t(M)$. Hence, $\height (I_t(M))=\height (P_t)=m_t$. By the generic perfection \cite{HE71}, $I_t(M)$ is a perfect ideal of $S$, and $S/I_t(M)$ is a Cohen-Macaulay ring. Hence,
	 all associated primes of $I_t(M)$ have the same height \cite{HE71} and $Q_t$ does not exist, showing that $I_t(M)=P_t$.  Hence $I_t(M)$ is a prime ideal and
	$$\height (I_t(M)^e)=\height( P_t)=m_t.$$
	This finishes the claim.
	
	Now let $r=mns$ be the number of the new indeterminates ${\bf u}=u_{ijk}$. By Lemma  \ref{height of determinantal ideal} we get for generic $a\in K^r$, 
	$\height (I_t(M_a))=\height (I_t(M)^e)=\min\{s, m_t\}$,
 and if $m_t<s$ then
$K[ {\bf x}] / I_t(M_a)$ is Cohen-Macaulay.
	
Let $K$ be algebraically closed and $t$ such that $m_t<s$. To see that $I_t(M_a)$ is a prime ideal for generic $a$ note that $I_t(M)^e$ is a prime ideal in $K(\bf u) [\bf x]$. In fact, let $fg \in I_t(M)^e$ 
with $f, g \in K(\bf u)[\bf x]$. Clearing denominators, we may assume $f, g \in K[\bf u, \bf x]$. There exists $0 \neq h(u) \in K[\bf u]$ s.t. $hfg \in I_t(M)$. Since  $h(u) \notin I_t(M)$ (otherwise $1 \in I_t(M)^e$, a contradiction) we get $fg \in I_t(M)$ and hence $f$ or $g$ is in $I_t(M)$, showing that $I_t(M)^e$ is prime.  Now we get the result from  \cite[Proposition 3.5 ]{NT99}. \end{proof}

\begin{Problem}\rm {\label{pr2.6}}
Does an analogous statement as in Theorem \ref{generic determinantal ideals} hold for matrices with entries in $K[[\bf x]]$ instead of $K[\bf x]$? Using Remark \ref {specialization} with a new definition of specialization in power series rings, we expect this for $K$ uncountable and $a$ outside the union of countably many nowhere dense subvarieties.  
\end {Problem}\rm 

The above proof does not work for power series, since the straightforward definition of specialization $I_a:=\{f(a,{\bf x})\mid f({\bf u},{\bf x})\in I\cap K[{\bf u}][[{\bf x}]]\}$ for an ideal $I$ in  $K({\bf u})[[{\bf x}]]$ may be 0 for all $a$ even if $I \neq 0$, as the following example shows. This example is due to Osgood and was also used by Gabrielov in his counter example to the nested approximation theorem in the analytic case (cf. \cite {Ro13}):

\begin {Example} \label{osgood} \rm{
Consider the morphism 
\[ \hat{\varphi} : \C[[u,x_1, x_2]] \to \C[[y_1,y_2]], \ u \mapsto y_1, \, x_1  \mapsto y_1y_2, \, x_2 \mapsto y_1y_2 \cdot \exp(y_2),\]
and let  $\varphi : \C(u)[[x_1, x_2]] \to \C[[y_1,y_2]]$ be given by the same assignment.
It is shown in \cite {Os16} that $\ker (\hat{\varphi}) = 0$. However, $I:= \ker(\varphi) \neq 0$ since it contains $x_2 - x_1  \cdot \exp(x_1/u)$, while $I \cap  \C[u][[x_1, x_2]] \subset \ker (\hat{\varphi})=0$} and hence $I_a =0$ for all $a \in \C$.
\end {Example}

\begin{Remark} \label {specialization} \rm
The problem is, that elements in the ring $K({\bf u})[[{\bf x}]]$ may have infinitely many denominators. A reasonable definition for a specialization in this ring is
$$I_a := \{f(a,{\bf x})\ | \  f({\bf u},{\bf x})\in I \text {, no denominator of } f \text { vanishes at } a\}. $$ 
	With this definition, which can be easily extended to finitely generated submodules of  $(K({\bf u})[[{\bf x}]])^p$, many properties of $I$ hold also for $I_a$ (e.g. the Hilbert-Samuel functions coincide) if $a$ is contained in the complement of countably many closed proper subvarieties of $K^r$.  For this to be useful we must assume that $K$ is uncountable.
	We do not pursue this here, since we need only the specialization  for ideals $I \subset K({\bf u})[{\bf x}]$.
\end{Remark} 

\medskip
The following theorem provides a necessary condition for finite determinacy for matrices with entries in $R = K[[x_1, \ldots, x_s]]$ with respect to the group $G$.

\begin{Theorem}{\label{height}}
	Let $A=[a_{ij}]\in \mathfrak{m}\cdot Mat(m,n,R)$ be finitely $G$-determined. Then the following holds:
	\begin{enumerate}
		\item\label{height is expected} Let  $K$ be infinite. Then we have for $t\in\{1, \ldots, n\}$   with $m_t:=(m-t+1)(n-t+1)$
		\[\height (I_t(A))=\min\{ s, m_t\}.\]
		
		\item\label{height 2} For any $K$ the following holds.
		
		(i) If $ s \geq mn$ then $\height(I_1(A))=mn$, i.e. $\{a_{ij}\}$ is an $R$-sequence. 
		
		(ii) If $s\le mn$ then $I_1(A)\supset \mathfrak{m}^k$ for some positive integer $k$, i.e. the entries of $A$ generate an $\mathfrak{m}$-primary ideal in $R$.
	\end{enumerate}
\end{Theorem}

\begin{proof}
	1. Assume that $A$ is $G$ $l$-determined and let $N \geq l+1$. Applying Theorem \ref{generic determinantal ideals} to $A_0=jet_l(A)$ (the truncation of the entries of $A$ after degree $l$ such that $A_0$ has polynomial entries and is $G$-equivalent to A), there is a matrix $B=[g_{ij}]\in Mat(m, n, K[{\bf x}])$ with entries of the form
	$$g_{ij}=\sum\limits_{k=1}^s c_{ijk}x_k^{N},\hskip 5pt c_{ijk}\in K,$$
    such that for all $t\in \{1,\ldots, n\}$ we obtain
	$$\height I_t(A_0+B)=\min\{s, m_t\}.$$
	
	The same equality holds
	for the extension of the ideal $I_t(A_0+B)$ in $R$ since the morphism $K[{\bf x}]\hookrightarrow R$ is flat. 
	Now, since $A\mathop\sim\limits^{G}A_0\mathop\sim\limits^{G}A_0+B$, we have 
	$\height I_t(A)=\height I_t(A_0)=\height I_t(A_0+B)$, since Fitting ideals are invariant under $G$-equivalence.

	2. This holds for any $K$ since 	
	we do not need Theorem \ref{generic determinantal ideals}  (where $K$ infinite was used) to show that
	\[\height(I_1(A))=\min\{s, mn\}.\]	
In fact, in case (i) we choose the entries of the matrix $B=[g_{ij}]$ to be $x_1^{N}, \ldots, x_{mn}^{N}$ and in case (ii) we choose the
first $s$ entries of the matrix $B$ to be $x_1^{N}, \ldots, x_s^{N}$ and the remaining entries 0, with $N$ sufficiently big. Then $\height \left(I_1(A_0+B)\right) = \min\{s, mn\}$, which can be seen for arbitrary $K$ as follows. Choose a global degree ordering on the variables, such that  $x_1^{N}, \ldots, x_{mn}^{N}$ in case (i) and $x_1^{N}, \ldots, x_s^{N}$ in case (ii) generate the leading ideal of $I_1(A_0+B)$.  By \cite [Corollary 5.3.14]{GP07} the dimension of  $K[{\bf x}]/I_1(A_0+B)$ is $s-mn$ in case (i) and $0$ in case (ii). In case (i)  the entries of $A_0+B$ are a regular sequence and generate a complete intersection. This is unmixed and therefore $\dim(R/I_1(A_0+B))=s-mn$. In case (ii) we get $\dim(R/I_1(A_0+B))=0$.
Since $A\mathop\sim\limits^{G}A_0+B$ we get $\height \left(I_1(A)\right)=\min\{mn,s\}$ and hence the result.
\end{proof}

\begin{Remark}\rm
	The above necessary condition is of course not sufficient. For example,  in any characteristic $f=x^k\in K[[x,y]]$ is not finitely contact determined by Theorem \ref{hypersurface} but $\height\left(\langle f\rangle\right)=1.$
\end{Remark}

\section{A finite determinacy criterion for column matrices}\label{special case}
Theorem \ref{height}  shows that finite $G-$determinacy of matrices in $M_{m,n}$ is rather restrictive. A  criterion which is at the same time necessary and sufficient for finite $G$-determinacy  for arbitrary $m, n$ in positive characteristic is unknown to us. 
 In this section we prove such a criterion for  1-column matrices in Theorem \ref{column matrix}. 
 
For a matrix $A=[a_1\hskip 4pt a_2\hskip 4pt\ldots\hskip 4pt a_m]^T\in M_{m,1},$ 
we denote by 
$Jac(A):=\left[\frac{\partial a_i}{\partial x_j}\right]\in M_{m,s}$
the Jacobian matrix of the vector $(a_1,\ldots, a_m)\in R^m$ and call it the Jacobian matrix of $A$. The extended tangent image has then the following concrete description as submodule of $R^m = M_{m,1}$:
		\[\tilde T_A^e\left(GA\right) =  IR^m+\left(\frac{\partial a_i}{\partial x_j}\right)\cdot R^s,\]
where $I=\langle a_1,\ldots, a_m\rangle $ is the ideal in $R$ generated by the entries of $A$ and		 $\left(\frac{\partial a_i}{\partial x_j}\right)\cdot R^s$ is the $R$-submodule of $R^m$ generated by the columns of $Jac(A)$.

\begin{Lemma}\label{isolated and presentation matrix}
	Let $A=[a_1\hskip 4pt a_2\hskip 4pt\ldots\hskip 4pt a_m]^T\in Mat(m, 1, \mathfrak{m})$ be such that $ m\le s$, where  $s$ is the number of variables. Let $\Theta_{(G,A)}$ be a presentation matrix of $M_{m,1}/\tilde T^e_A(GA)$, where $\tilde T^e_A(GA)$ is the extended tangent image at $A$ to the orbit $GA$. Then 
	\[\sqrt{I_1(A)+I_m(Jac(A))}= \sqrt{I_m\left(\Theta_{(G,A)}\right)} 
	= \sqrt{\Ann_R(M_{m,1}/\tilde T^e_A(GA))}.\] 
In particular, $\dim_K (M_{m,1}/\tilde T^e_A(GA)) < \infty$ iff  $\dim_K\big(R/I_1(A)+I_m(Jac(A))\big)< \infty$.
\end{Lemma}

\begin{proof}
	A presentation matrix of $M_{m,1}/\tilde T^e_A(GA)$ is the following two-block matrix of size $m\times( m^2+s)$	
	\[\Theta_{(G,A)}=
	\left[ {\begin{array}{*{20}c}
			{\frac{{\partial a_{1} }}{{\partial x_1 }}} & {\frac{{\partial a_{1} }}{{\partial x_2 }}} & {\ldots} & {\frac{{\partial a_{1} }}{{\partial x_s }}}  \\
			{\frac{{\partial a_{2} }}{{\partial x_1 }}} & {\frac{{\partial a_{2} }}{{\partial x_2 }}} & {\ldots} & {\frac{{\partial a_{2} }}{{\partial x_s }}}  \\
			{\vdots} & {\vdots} & {\ldots} & {\vdots}  \\
			{\frac{{\partial a_{m} }}{{\partial x_1 }}} & {\frac{{\partial a_{m} }}{{\partial x_2}}} & {\ldots} & {\frac{{\partial a_{m} }}{{\partial x_s }}}  \\
			\end{array}} \right.\left| {\begin{array}{*{20}c}
				\vec{a} & \vec{0} & \vec{0} & {\ldots}  & \vec{0}\\
				\vec{0} &  \vec{a} & \vec{0} & {\ldots} & \vec{0}  \\
				{\vdots} & {\vdots} & {\vdots} & {\ldots} & {\vdots} \\
				\vec{0}  & \vec{0} & \vec{0} & {\ldots}  & \vec{a}  \\
				\end{array}} \right],
	\]	
	where $\vec{a}:=(a_1,\ldots, a_m)$ and $\vec{0}:=(0,\ldots,0)\in K^m$. Since $m\le s$, it is obvious that
	\[I_m\left(\Theta_{(G,A)}\right)\subset \langle a_1,\ldots, a_m\rangle + I_m(Jac(A)).\]
On the other hand we have the inclusion
	\[\sqrt{\langle a_1,\ldots, a_m\rangle+I_m(Jac(A))}\subset
	\sqrt{\sqrt{\langle a_1,\ldots, a_m\rangle}+\sqrt{I_m(Jac(A))}}\subset\sqrt {I_m(\Theta_{(G,A)})}.\] 
	since  $\langle a_1,\ldots, a_m\rangle^m$ is the ideal generated by $m\times m$ minors of the right-hand block of $\Theta_{(G,A)}$ and $Jac(A)$ is a block of $\Theta_{(G,A)}$. The rest is well-known (see also \cite[Proposition 4.2]{GP16}).
\end{proof}

We show now that there exist finitely $G$-determined matrices in $Mat(m,1, R)$ with entries of arbitrary high order.

\begin{Proposition}\label{example isolated}
	Let $R=K[[x_1, \ldots, x_s]]$ and $M_{m,1}=Mat(m,1, R)$ with  $m \leq s$. Let $\characteristic(K)=p \geq 0$, $N\ge 2$ an integer and if $p>0$ let $p\nmid N$. Assume there are $c_{ij} \in K, i=1,\ldots, m, j=1,\ldots, s$, such that no maximal minor $m_1, \ldots, m_r$ of the $m\times s$ matrix $[c_{ij}]_{i=1,\ldots, m, j=1,\ldots, s}$ vanishes (which is always possible if $K$ is infinite or if $K$ is arbitrary and $m=1$). Set	
\[ f_i  := c_{i1}x^N_1+\cdots+c_{is}x^N_s, \ \  i=1,\ldots, m.  \] 
	 Then $A :=\left[f_1 \dots f_m\right]^T$ is finitely $G$-determined.
	 
	\end{Proposition}
	
	\begin{proof}
	 
	 Let $J$ be the ideal of $K[{\bf x}]$ generated by $f_1, f_2,\ldots, f_m$ and all $m\times m$ minors of $Jac(A)$. We claim that $V(J)=\{0\}$ in $\bar{K}^s$, with $\bar{K}$ an algebraic closure of $K$. Indeed, let ${\bf a}=(a_1,a_2,\ldots,a_s)\in V(J)$. Then at least $s-m+1$ components of ${\bf a}$ must  be zero, since  ${\bf a}$ is a zero of  all products $x^{N-1}_{j_{1}} \cdot x^{N-1}_{j_{2}}\cdots x^{N-1}_{j_{m}}$, where $j_i\in\{1,2,\ldots,s\}$ for all $i=1,\ldots,m$ and $j_i\ne j_k$ for all $i\ne k$, and $m \leq s$. Without loss of generality we assume that the last $s-k$ components of ${\bf a}$ are zero for some $k\le m-1$.  Then $f_i({\bf a}) = c_{i1}a_1^N+c_{i2}a_2^N+\cdots +c_{ik}a_k^N.$
	 
	 Consider the homogeneous  system of $m$ linear equations in $k$ variables $y_1,\ldots,y_k$ 
	 $$(H): \sum\limits_{j=1}^{k}c_{ij}y_j=0, \hskip 10pt i=1,\ldots,m.$$
	 Since ${\bf a}$ is also a zero of $f_1,\ldots, f_m$, it follows that $y_j=a^N_j$, $j=1,\ldots,k$, is a solution of $(H)$.
	 By the choice of $c_{ij}$ there must be a non-zero $k\times k$ sub-determinant of the coefficient matrix of $(H)$. This implies that $(H)$ has only the trivial solution. Therefore, ${\bf a}={\bf 0}$  and the claim follows.  As a consequence, $\dim(K[{\bf x}]/J)=0$ and hence $\dim(K[[{\bf x}]]/J\cdot K[[{\bf x}]])=0$. Applying Lemma \ref{isolated and presentation matrix} and \cite[Proposition 4.2.5]{GP16}, $A$ is finitely $G$-determined.
\end{proof}

\begin{Example}\rm
$K^{ms}\smallsetminus V(m_1\cdot m_2\cdots m_r) =\emptyset$ may happen for finite $K$. If $K=\{0,1\}, m=2$ and $s=4$, then it is easy to see that at least one of the six 2-minors of  $[c_{ij}]_{i=1,\ldots, 2, j=1,\ldots, 4}$ is 0 for any choice of $c_{ij} \in K$. However, if $m=1$ we can choose $c_{1j} =1$ for all $j$ and then $A=[f_1]$ is finitely $G-$determined for any $K$.
\end{Example}

We need the semi-continuity of the $K$-dimension of a 1-parameter family of finitely generated modules over a power series ring. This is well known for complex analytic power series by the finite coherence theorem. But since we could not find a reference for our situation, we give a proof here. 

Let  $P=K[t][[{\bf x}]]$,
${\bf x}=(x_1,\ldots,x_s)$, $K$ an arbitrary field, and $M$ a finitely generated $P$-module.  For
$t_0\in K = \A^1$, set
\begin{align*}
M(t_0):&=M\mathop\otimes\limits_{K[t]}\ (K[t]/\langle t-t_0\rangle)
\cong M/\langle t-t_0\rangle\cdot M,\\
& \mathfrak{m}_{t_0}:=\langle x_1,\ldots, x_s, t-t_0\rangle\subset P.
\end{align*}
We remark that $M(t_0) \cong M_{\mathfrak{m}_{t_0}}\big/\langle t-t_0\rangle\cdot M_{\mathfrak{m}_{t_0}}$.

\begin{Proposition}\label{semi-continuity}
	With the above notations, for any  $o \in \A^1$ there is an open neighborhood $U$ of $o$ such that for all $t_0\in U$, we have
	\[\dim_KM(t_0)\le\dim_KM(o).\]
\end{Proposition}

\begin{proof}
	Without loss of generality we may assume $o=0$ and that $\dim_K M(0)<\infty.$ Then $M$ is quasi-finite but in general not finite over $K[t]$.  However, we show that the restriction of $M$ to some open subset of $\A^1$ is in fact finitely generated  over $K[t]$ (\textit {Step 4}), so that we can apply the semicontinuity of the rank of the presentation matrix of $M$ as $K[t]$--module.
	 The first steps in the proof are used to show that we can reduce to this case. We may assume that $K$ is infinite, since otherwise $\{0\}$ is open and the statement is trivially true.

Consider a primary decomposition of $\annihilator_P(M)$,
	\[\annihilator_P(M)=\mathop\cap\limits_{i=1}^rQ_i\subset P.\]
For all $i=1,\ldots, r$ let $\bar{Q_i}$ denote the image of $Q_i$ under the morphism $P\to P/\langle x_1,\ldots, x_s\rangle\cong K[t]$.

\noindent{\bf\textit {Case 1:}} $\variety(\langle x_1,\ldots, x_s\rangle)\not\subset \variety(Q_i)$ for all $i=1,\ldots,r$ in $\spectrum(P)$, i.e. $Q_i\not\subset\langle x_1,\ldots, x_s\rangle$ for all $i=1,\ldots,r$. In this case $\bar{Q_i}=\langle f_i\rangle\subset K[t]$ for some $f_i \neq 0$. Since
	\[U:=\A^1\smallsetminus\mathop\cup\limits_{i=1}^r\variety(f_i)\ne\emptyset,\]
for all $t_0\in U$ we have  
	$M_{\mathfrak{m}_{t_0}}=0$ and hence $\dim_K M(t_0)=0\le \dim_KM(0).$ 
	
\noindent{\bf\textit {Case 2:}} $\variety(\langle x_1,\ldots, x_s\rangle)\subset \variety(Q_i)$ for some $i\in\{1,\ldots,r\}$, i.e. $Q_i\subset\langle x_1,\ldots, x_s\rangle$.
	
	\textit{Step 1:} For $Q_i\subset\langle x_1,\ldots, x_s\rangle$ we have that $Q_i$ is unique and
\begin{align*}
	\sqrt{Q_i}=\langle x_1,\ldots, x_s\rangle. \tag{3.1}\label{21}
\end{align*}
	Indeed, we have
	$\dim_K M_{\mathfrak{m}_0}/\langle t\rangle\cdot M_{\mathfrak{m}_0}=\dim_KM(0)<\infty$, and $M_{\mathfrak{m}_0}$ is finitely generated over the local ring $P_{\mathfrak{m}_0}$. Therefore, by Krull's principal ideal theorem (\cite [Theorem B.2.1]{GLS07}) $\dim (M_{\mathfrak{m}_0})\le 1,$  which implies
	$\dim P_{\mathfrak{m}_0}/(\sqrt{ Q_{i}})_{\mathfrak{m}_0}=1.$
Therefore
	\[\sqrt{ Q_{i}}=\langle x_1,\ldots, x_s\rangle\]
and $\sqrt{ Q_{i}}$ is a minimal associated prime of $\annihilator_P M$. Hence $Q_i$ is unique.
	
	{\textit {Step 2:}}  Let, by \textit{Step 1}, $Q$
	be the only primary component of $\annihilator_PM$ contained in $\langle x_1,\ldots, x_s\rangle$. We set  
	$$\bar M := M/Q \cdot M$$
and  $\bar M(t_0):= \bar M\mathop\otimes\limits_{K[t]}K[t]/\langle t-t_0\rangle$ for $t_0\in K$.
Then, for $t_0=0$ we have
	\begin{align*}
	\dim_K \bar M(0)\le \dim_K M(0)\tag{3.2}\label{2}.
	\end{align*}
	
{\textit{Step 3:}}  Set 
	$W:=\variety(\bar Q)\smallsetminus (\mathop\cup\limits_{Q_i\ne Q}\variety(\bar{Q_i})) =\A^1\smallsetminus (\mathop\cup\limits_{Q_i\ne Q}\variety(\bar{Q_i}))\ne\emptyset.$	
We claim that for $t_0\in W$ 
	\begin{align*}
	\dim_K\bar M(t_0)=\dim_KM(t_0)\tag{3.3}\label{4}.
	\end{align*}
	
In fact, since $\mathfrak{m}_{t_0}\in \variety(Q)\smallsetminus \mathop\cup\limits_{Q_i\ne Q}\variety(Q_i)$,
	it follows that $(Q\cdot M)_{\mathfrak{m}_{t_0}}=0.$
	Hence, there is a $P$-module homomorphism
	$\varphi: \bar M_{\mathfrak{m}_{t_0}}\cong M_{\mathfrak{m}_{t_0}}/(Q\cdot M)_{\mathfrak{m}_{t_0}}=M_{\mathfrak{m}_{t_0}}$ with  
	$\varphi\left(\langle t-t_0\rangle\cdot \bar M_{\mathfrak{m}_{t_0}}\right)=\langle t-t_0\rangle\cdot M_{\mathfrak{m}_{t_0}}$, implying the claim.
	
{\textit{Step 4:}} 
	By \eqref {21} $\sqrt{Q}=\langle x_1,\ldots, x_s\rangle$, hence $P/Q$ is a finitely generated $K[t]$-module. Thus  $M/Q\cdot M$ is a finitely generated $K[t]$-module, having a presentation
	\[K[t]^m\stackrel{\psi(t)}\longrightarrow K[t]^n\to \bar M\to 0.\]
    This implies for $t_0\in K$ the exact sequence 
	$K^m\stackrel{\psi(t_0)}\longrightarrow K^n\to \bar M(t_0)\to 0,$
	which yields
	$\dim_K\bar M(t_0)=n-\rank \psi(t_0).$
	Since $\rank \psi(t)$ is lower semi-continuous on $\A^1$, there is an open neighborhood $V$ of $0$ in $\A^1$ such that for all $t_0\in V$ we get the inequality
	\begin{align*}
	\dim_K\bar M(t_0)\le \dim_K\bar M(0).\tag{3.4}\label{5}
	\end{align*}
	
With $U=W\cap V$ the assertion of the proposition follows from \eqref{4}, \eqref{5}, and \eqref{2}.
\end{proof}

We prove now our main result of this section.
\begin{Theorem}\label{column matrix}
	Let $A=[f_1 \ldots f_m]^T\in  \mathfrak{m}\cdot M_{m,1}$,  $m \geq 1$. For $1<m<s$ assume $K$ to be infinite. Then the following are equivalent:
	\begin{enumerate}
		\item $A$ is finitely $G$-determined.
		\item $\dim_K\left(M_{m,1}\big/\tilde T^e_A(GA)\right)=:d_e<\infty.$
	\end{enumerate}
	$K$ infinite is not needed for 2 $\ \Rightarrow$ 1. Moreover, if condition 2 is satisfied then $A$ is $G$ $(2d_e-\order(A)+2)$-determined.
\end{Theorem}

\begin{proof}
	2. $\Rightarrow$ 1.  is a consequence of Theorem \ref{GP16}, as well as 1. $ \Rightarrow$ 2. for char($K$)= 0. 
	
	Let $\characteristic(K)=p>0$ and assume that $A$ is $G$ $k$-determined.\\
{\bf{\textit {Case 1: $m < s$}}}.
	By finite determinacy we may assume that $A=[f_1 \ldots f_m]^T$ is a matrix of polynomials.  Let $N\in \N$ be such that $N>k$ and $p\nmid N$. Let $B=[g_1 \ldots g_m]^T\in Mat(m,1,\mathfrak{m})$, where for $i=1,\ldots,m$,
	\[g_i=c_{i1}x^N_1+c_{i2}x^N_2+\cdots+c_{is}x^N_s\]
	and $c_{ij}\in K$ as in Proposition \ref{example isolated}  (which is possible by assumption). Consider 
	\[B_t=B+tA=[g_1+tf_1\hskip 5pt\ldots\hskip5pt g_m+tf_m]^T\in Mat(m,1, K[t][{\bf x}]).\]
	Let $Q_t$ be the ideal of $K[t][[{\bf x}]]$ generated by the entries of $B_t$ and all  $m\times m$ minors of the matrix $\left[\frac{\partial(g_i+tf_i)}{\partial x_j}\right]$. Then by Proposition \ref{semi-continuity}, there exists a neighborhood $U\subset \A^1_K$ of $0$ such that for all $t_0\in U$,
	\[\dim_K\left(K[[{\bf x}]]/Q_{t_0}\right)\le\dim_K\left(K[[{\bf x}]]/Q_0\right)<\infty,\]
	where the second inequality follows from the proof of Proposition \ref{example isolated}.
	By Lemma \ref{isolated and presentation matrix} we get that also
	$\dim_K\left(M_{m,1}\Big/\tilde T^e_{B_{t_0}}(GB_{t_0})\right)<\infty$.
	Let $t_0\in U$, $t_0\ne 0$. Since $B_{t_0}\mathop\sim\limits^{G}t_0A\mathop\sim\limits^{G}A$ we obtain 
	\[\dim_K\left(M_{m,1}\Big/\tilde T^e_{ A}(G A)\right)=\dim_K\left(M_{m,1}\Big/\tilde T^e_{B_{t_0}}(GB_{t_0})\right)<\infty.\]
{\bf{\textit {Case 2: $m \ge s$}}}. By Theorem \ref{height}.\ref{height 2} (ii),  $I=\langle f_1,\ldots, f_m\rangle $ is $\mathfrak m$-primary and since  $IR^m \subset \tilde T^e_{ A}(G A)$ the claim follows.
\end{proof}

The following corollary shows a criterion for finite $G$-determinacy of a column matrix, which does not use $R-$modules but only ideals in $R =K[[x_1,\ldots,x_s]]$. In addition, another determinacy bound is provided.

\begin{Corollary}\label{main corollary}
		Let  $A=[a_1 \ldots a_m]^T\in Mat(m, 1, \mathfrak{m})$.
		\begin{enumerate}
		 \item\label{main 1} If $m\ge s$
		 then $A$ is finitely $G$-determined if and only if there is some integer $k\ge 0$ such that
		 \begin{align*}
		 \mathfrak{m}\cdot\langle a_1,\ldots, a_m\rangle\supset \mathfrak{m}^{k+2},
		 \end{align*}
		 i.e. $I_1(A)$ is primary. $A$ is then $(2k-\order(A)+2)$-determined.
		 \item \label{main 2} Let $m\le s$ and for $m >1 $ assume that $K$ is infinite. Then $A$ is finitely $G$-determined if and only if there is some integer $k\ge 0$ such that
		\begin{align*}
		\langle a_1,\ldots, a_m\rangle + I_m(Jac(A))\supset \mathfrak{m}^k.
		\end{align*}
		Furthermore, $A$ is then $G$ $(2km-\order(A)+2)$-determined.
		\end{enumerate}
\end{Corollary}
\begin{proof}

1. 	If $A$ is finitely $G$-determined, Theorem \ref{height}.\ref{height 2}\ (ii) implies the inclusion $ ``\supset"$. If  $``\supset"$ holds, we have
	\[\mathfrak{m}^{k+2}\cdot M_{m,1}\subset \mathfrak{m}\cdot I_1(A)\cdot M_{m,1}+\mathfrak{m}^{2}\cdot\left\langle\frac{\partial A}{\partial x_1},\ldots,\frac{\partial A}{\partial x_s} \right\rangle= \mathfrak{m}\cdot \tilde T_A(GA).\] 
	By Theorem \ref{GP16},  $A$ is $G$ $(2k-\order(A)+2)$-determined.
	
2. The characterization of finite determinacy follows from Theorem \ref{column matrix} and Lemma \ref{isolated and presentation matrix}. 
	For the determinacy bound, we have $$\mathfrak{m}^{km}\subset\langle a_1,\ldots, a_m\rangle^m + I_m(Jac(A))\subset I_m(\Theta_{(G,A)}),$$ 
	where the second inclusion follows from the proof of Lemma \ref{isolated and presentation matrix}. By \cite[Proposition 4.2]{GP16} we obtain the result.
	\end{proof}
\section{Contact equivalence of ideals} \label{complete intersections}

By associating to a matrix $A=[a_1 \ldots a_m]^T\in M_{m,1}$  the ideal $I=\langle a_1,\ldots, a_m\rangle\subset R$ generated by the entries of $A$, the results of the previous section about $G$-equivalence for matrices in $M_{m,1}$ can easily be transferred to contact equivalence for ideals in $R$. The minimal number of generators of $I$ is denoted by $\mng(I)$.
Recall that two ideals $I$ and $J$ of $R$ are called {\textit {contact equivalent}}, denoted $I\mathop  \sim \limits^c J$, if $R/I \cong R/J$ as local $K$-algebras. 

\begin{Definition} \label{determined}
Let  $I$ be a proper ideal of $R$ and  $a_1, \ldots, a_m$ a minimal set of generators of $I$.
$I$ is called {\bf{contact $k$-determined}} if for every ideal $J$ of $R$ that can be generated by $m$ elements $b_1, \ldots, b_m$ with $b_i-a_i\in\mathfrak{m}^{k+1}$ for $i=1,\ldots,m$, we have $I\mathop\sim\limits^c J$. 
$I$ is called {\bf {finitely contact determined}} if $I$ is contact $k$-determined for some $k$. 

\end{Definition}

The notion of contact determinacy  is independent of the minimal set of generators of $I$ and hence an intrinsic property of $I$. In fact, let $I$ be contact $k$-determined w.r.t. $a_1, \ldots, a_m$ and let
$a'_1, \ldots, a'_m$ be another minimal set of generators of $I$ and denote by $A$ resp. $A'$ are the corresponding column matrices. If an ideal $J$ can be generated by $b_1, \ldots, b_m$, satisfying $b_i - a'_i \in\mathfrak{m}^{k+1}$, we have to show that $I\mathop\sim\limits^c J$:
By Lemma \ref {Mather 1} below, $A=U \cdot A'$ with $U\in GL(m, R)$. Let $b'_i$ be the entries of $B' = U \cdot B$, $B$ the matrix with entries $b_i$ then $ B'-A  = U\cdot (B - A') \in \mathfrak{m}^{k+1}M_{m,1}$ and $J$ is generated by $b'_1, \ldots, b'_m$ with $b'_i - a_i \in\mathfrak{m}^{k+1}$. Hence $I\mathop\sim\limits^c J$ and $I$ is contact $k$-determined w.r.t. $a'_1, \ldots, a'_m$.



The following lemma is proved in \cite[2.3 Proposition]{Mat68} for map-germs, but the proof works also in our more general situation.

\begin{Lemma}\label{Mather 1}
	 Let  $I$ be an ideal of a unital Noetherian local ring $Q$ generated by two lists of $m$ elements, say $\langle a_1,\ldots, a_m\rangle$ and $\langle b_1\ldots, b_m\rangle$. Let $A=[a_1\hskip 3pt a_2\hskip 3pt\ldots\hskip 3pt a_m]^T$ and $B=[b_1\hskip 3pt b_2\hskip 3pt\ldots\hskip 3pt b_m]^T$ be the corresponding column matrices. Then there is an invertible matrix $U\in GL(m, Q)$ such that $B=U\cdot A$.
\end{Lemma}

Note that for 1--column matrices  
left-right equivalence (i.e.,  $G$-equivalence) and left equivalence coincide.

\begin{Proposition}{\label{contact equivalence}}
	Let $I$ and $J$ be proper ideals of $R$ given by the same number of $m$ generators
	and let $A$ resp. $B$ be the corresponding matrices in $M_{m,1}$.
	Then the following are equivalent:
	\begin{enumerate}
		\item $I\mathop  \sim \limits^c J$.
		\item  There is an automorphism $\phi\in Aut(R)$ such that $\phi(I)=J$.
		\item There are  $\phi\in Aut(R)$ and $U\in GL(m,R)$ such that $B=U\cdot \phi(A)$.
		\item $A\mathop\sim\limits^G B$.
	\end{enumerate}
	Moreover, if $m = \mng(I)$, then $I$ is contact $k$-determined iff $A$ is $G$ $k$-determined.
\end{Proposition}

\begin{proof}
The equivalence of 1. and 2. follows from the lifting lemma \cite[Lemma 1.23]{GLS07}, that of 2. and 3.  from  Lemma \ref{Mather 1}. The equivalence of 3. and 4. is obvious. The last statement follows from the equivalence of 1. and 4. and from the definition of determinacy.
\end{proof}

%
%

\begin{Definition}\label{definition tjurina}
	Let $I$ be a proper ideal of $R$ with $\mng(I)=m$ and $A\in M_{m,1}$ the column matrix corresponding to a minimal set of generators of $I$. We define
	\[T_I:=M_{m,1}{\Big/}\left(I\cdot M_{m,1}+\left\langle\frac{\partial A}{\partial x_1},\ldots,\frac{\partial A}{\partial x_s} \right\rangle\right)\]
	and set $\tau(I):=\dim_KT_I.$
\end{Definition}

\begin{Remark}\label{Tjurina}\rm
We have 
		$T_I=M_{m,1}\big/\tilde T_A^e\left(GA\right) = R^m\Big/IR^m+\left(\frac{\partial a_i}{\partial x_j}\right)\cdot R^s.$
%
 By Proposition \ref{semi-continuity} $\tau(I)$ is semicontinuous if we perturb the entries of $A$.

For a complete intersection the $R/I$-module $T_I$ is called the {\em Tjurina module of I} and $\tau(I)=\dim_KT_I$ the {\em Tjurina number of I} (see \cite[Theorem 1.16 and Definition 1.19]{GLS07} for the complex analytic case). The support of $T_I$ is then the singular locus of $I$.

\end{Remark}

\begin{Theorem}{\label{geomchar}}
	Let $I$ be a proper ideal of $R$. 
	\begin {enumerate}
	\item  If $\dim(R/I) = 0$ then $I$ is finitely contact determined.
	\item  If  $\dim(R/I) >0$ and $K$ is infinite, then the following are equivalent:\\	
	(i) $I$ is finitely contact determined.  	\\
	(ii) $\tau(I)<\infty$.	\\
	(iii)  $R/I$ is an isolated complete intersection singularity. 	\\
	If one of these condition is satisfied then $I$ is contact $(2\tau(I)-\order(I)+2)$-determined.
	Moreover, if $J$ is an ideal such that $R/J$ is a deformation of $R/I$, then $J$ is contact $(2\tau(I)+1)$-determined resp. $2\tau(I)$-determined if $\order(J) \geq 2$.
	\end{enumerate}
\end{Theorem}

\begin{proof}
	Let $A$ be the column matrix corresponding to a minimal set of $m$ generators of $I$. 
	
	1. If $\dim(R/I) = 0$ then  some power of $\mathfrak m$ is contained in $I$ and $m = \mng(I) \geq s$. Then $A$ is finitely  $G$-determined by Corollary \ref{main corollary}.\ref{main 1} and hence $I$ is finitely contact determined by Proposition \ref{contact equivalence}. 
	
	2. The equivalence of (i) and (ii) as well as the determinacy bound follows from Theorem \ref{column matrix} without the assumption $\dim(R/I)>0$. 
	
	We prove the equivalence of (i) and (iii). If $I$ is finitely determined then $m < s$ by Theorem \ref{height}.\ref{height 2} (ii). By Theorem \ref{height}.\ref{height 2} (i), we get $\height(I_1(A))=m$, i.e. $\dim R/I = s-m$ and thus $R/I$ is a complete intersection. The isolatedness of $R/I$ and the converse direction (iii)  $\Rightarrow$ (i) are direct consequences of Corollary \ref{main corollary}.\ref{main 2}. 
		
	If $R/J$ is a deformation of $R/I$, $J$ defines an ICIS of the same dimension and hence $J$ is contact $(2\tau(J)+1)$-determined. By semicontinuity of $\tau$ (Remark \ref{Tjurina}) $\tau(I) \ge \tau(J)$, implying that $J$ is $(2\tau(I)+1)$-determined, resp. $2\tau(I)$-determined if ord($J$) $\geq 2$.
\end{proof}

 
\begin {Problem}\rm
 Note that (ii) $\Leftrightarrow$ (iii) $\Rightarrow$ (i) in Theorem \ref{geomchar} holds for any $K$, as well as (i) $\Rightarrow$ (iii) for hypersurfaces by Theorem \ref{hypersurface}. However, we do not know whether (i) $\Rightarrow$ (iii) holds for finite $K$ and a complete intersection which is not a hypersurface.
\end {Problem}


We complete the section with the hypersurface case, i.e. ideals generated by one element $f\in R=K[[{\bf x}]]$.  
In addition to contact equivalence we consider also {\em right equivalence} for $f\in R$, with $f$ being right equivalent to $g$ if $\phi(f) = g$ for some $\phi \in Aut(K[[{\bf x}]])$. $f$ is {\em right k-determined} if any $g$ with $f-g \in \mathfrak m^{k+1}$ is right equivalent to $f$.

Then the Tjurina number is 
$$\tau(f)=\dim_K K[[{\bf x}]]/\langle f, {\partial f}/{\partial x_1}, \ldots, {\partial f}/{\partial x_s}\rangle$$
and the Milnor number is 
$$\mu(f)=\dim_K K[[{\bf x}]]/\langle {\partial f}/{\partial x_1}, \ldots, {\partial f}/{\partial x_s}\rangle.$$

\begin{Theorem}\label{hypersurface}
Let $K$ be any field and $f\in \mathfrak{m} \subset R$. 
\begin{enumerate}
\item $f$ is finitely contact determined iff $\tau(f)$ is finite. 
\item $f$ is finitely right determined iff $\mu(f)$ is finite.
\end{enumerate}
Moreover, if  $\tau(f) < \infty$ (resp. $\mu(f) < \infty$), then $f$ is contact $(2\tau(f)-\order(f)+2)$--determined
(resp. right $(2\mu(f)-\order(f)+2)$--determined). 
\end{Theorem}

\begin{proof}
Since the statement is obviously true for $f=0$ we may assume that $f \neq 0$. For contact equivalence, the statement follows from Theorem \ref{column matrix} with $m=1$ (for this case it was not assumed that $K$ is infinite). For right equivalence, if $ \mu(f) < \infty$, then  $f $ is finitely right determined by \cite[Theorem 3.2]{GP16} for any $K$. The proof of the converse goes as for contact equivalence.
\end{proof}

\begin{Remark}\rm
Theorem \ref{hypersurface} was already stated  in \cite[Theorem 5]{BGM12}. However, the proof of the direction that finite determinacy implies the finiteness of $\tau$ resp. $\mu$ in \cite {BGM12} contains a gap 
since it assumes the orbit map $G^{(k)}\to G^{(k)}jet_k(f)$ to be separable. Separability of the orbit map holds always in characteristic 0  and quite often in positive characteristic, but not always (see \cite[Example 2.9]{GP16}). 
\end{Remark}

%


{\bf Acknowledgement:} We thank Andr\'e Galligo for suggesting Gabrielov's example and Ng{\^o}~V. Trung for useful discussions, in particular for providing a proof of Theorem \ref {generic determinantal ideals}. The second
author would like to thank the Abdus Salam International Centre for Theoretical Physics (ICTP) for support and the department of mathematics of the University of Kaiserslautern for its hospitality.


\providecommand{\bysame}{\leavevmode\hbox to3em{\hrulefill}\thinspace}
\providecommand{\MR}{\relax\ifhmode\unskip\space\fi MR }
\providecommand{\MRhref}[2]{%
	\href{http://www.ams.org/mathscinet-getitem?mr=#1}{#2}
}
\providecommand{\href}[2]{#2}

Fachbereich Mathematik, Universit\"at Kaiserslautern, Erwin-Schr\"odinger Str.,
67663 Kaiserslautern, Germany

\noindent E-mail address: greuel@mathematik.uni-kl.de
\medskip
\vskip 3pt 

\noindent Department of Mathematics, Quy Nhon University, 170 An Duong Vuong, Quy Nhon, Vietnam

\noindent Email address: phamthuyhuong@qnu.edu.vn
\end{document}